\documentclass[10pt]{article}
\usepackage{amssymb}
\usepackage{amscd}
\usepackage{amsmath}
\usepackage{amsfonts}
\usepackage{amsthm}
\usepackage{color}
\usepackage{enumerate}
\usepackage{enumitem}

\theoremstyle{plain}

\newtheorem{theo}{Theorem}[section] 
\newtheorem{prop}[theo]{Proposition}
\newtheorem{lemme}[theo]{Lemma}
\newtheorem{cor}[theo]{Corollary}

\theoremstyle{definition}

\newtheorem{rem}[theo]{Remark}

\newcommand{\Ci}{{\mathcal{C}}^{\infty}}

\newcommand{\R}{{\mathbb{R}}}
\newcommand{\N}{{\mathbb{N}}}
\newcommand{\C}{{\mathbb{C}}}
\newcommand{\Z}{{\mathbb{Z}}}

\newcommand{\be}{{\beta}}
\newcommand{\al}{{\alpha}}

\newcommand{\De}{{\Delta}}
\newcommand{\si}{{\sigma}}
\newcommand{\hb}{{\hbar}}

\newcommand{\te}{{\theta}}
\newcommand{\om}{{\omega}}
\newcommand{\Om}{\Omega}

\newcommand{\op}{\operatorname}

\newcommand{\bigo}{\mathcal{O}}    

\newcommand{\con}{\overline}

\newcommand{\Fa}{\mathcal{A}}
\newcommand{\To}{\mathcal{T}}
\newcommand{\Tosc}{\mathcal{T}_{\op{sc}}}

\newcommand{\sy}{\mathcal{S}} 

\newcommand{\synorsp}{\mathcal{E}}
\newcommand{\herm}{\op{Herm}}
\newcommand{\der}{\mathcal{D}}

\newcommand{\sis}{\sigma_{\op{s}}}
\newcommand{\sip}{\sigma_{\op{p}}}

\author{Laurent Charles \footnote{Sorbonne Universit{\'e}s, UPMC Univ Paris 06, UMR 7586, Institut de Math{\'e}matiques de Jussieu-Paris Rive Gauche, F-75005, Paris, France} } 
\title{Subprincipal symbol for Toeplitz operators}

\begin{document}

\maketitle

\begin{abstract} We establish some subprincipal estimates for Berezin-Toeplitz operators on symplectic compact manifolds. From this, we construct a family of subprincipal symbol maps and we prove that these maps are the only ones satisfying some expected conditions.

\end{abstract}

Toeplitz operators on symplectic manifolds are similar to semiclassical pseu\-do-differential operators on cotangent bundles. In particular, we can extend to Toeplitz operators the usual techniques to describe the spectrum of pseudo-differential operators, as for instance the trace formula \cite{BPU2} or the Bohr-Sommerfeld conditions \cite{oim_lag}, \cite{oim_demi}, \cite{Yohann}. Two important ingredients in these semi-classical results are the principal and subprincipal symbols of an operator. One issue is that there is no obvious definition for the subprincipal symbol of a Toeplitz operators, cf. as instance \cite{BPU2} or \cite{Bo2}. 

In this paper, we introduce some axioms that a subprincipal symbol should satisfy to our point of view. Then we construct all the subprincipal symbol satisfying these axioms. We work in the general setting introduced in \cite{oim_symp_quant} for the quantization of symplectic manifolds. General reference for the Toeplitz operators are  \cite{BoGu}, \cite{MaMa}, \cite{BoMeSc}, \cite{Gu}. Our construction is inspired from our previous work on K{\"a}hler manifolds, in particular \cite{oim_demi} and \cite{oim_eq}, and uses in an essential way the metaplectic correction.  

\section{Statement of the results} 

To start with, consider a compact K{\"a}hler manifold $M$ equipped with a holomorphic Hermitian line bundle $L$ and with a holomorphic Hermitian vector bundle $A$. Assume that $L$ is positive so that the Chern curvature of $L$ is $\frac{1}{i} \om $ where $\om \in \Om^2 (M, \R)$ is symplectic. For any $k \in \N$, let $\mathcal{H}_k$ be the space of holomorphic sections of $L^k \otimes A$. 

More generally, let $(M, \om)$ be any symplectic compact manifold such that $\frac{1}{2 \pi } [\om]$ is integral. Introduce a Hermitian line bundle $L \rightarrow M$ with a Hermitian connection $\nabla $ of curvature $\frac{1}{i} \om$, a Hermitian vector bundle $A$ and an almost complex structure $j$ compatible with $\om$. Then consider a family of finite dimensional subspace 
$$ \mathcal{H}_k \subset \Ci ( M , L^k \otimes A) , \qquad k \in \N^* $$ consisting of almost-holomorphic sections in the sense of \cite{oim_symp_quant}. In the sequel, we will refer to the case where $M$ is K{\"a}hler, $L$ and $A$ holomorphic, $\nabla$ is the Chern connection and $\mathcal{H}_k$ consists of holomorphic sections, as the K{\"a}hler case. Observe that $\mathcal{H}_k$ has a natural scalar product obtained by integrating the pointwise scalar product against the Liouville measure $\om^n / n!$.

To these data is associated a star-algebra $\To$ consisting of the so-called Berezin-Toeplitz operators or more briefly Toeplitz operators. The definition will be recalled later, let us just say for now that a Toeplitz operator $T$ is an endomorphism family $(T_k : \mathcal{H}_k \rightarrow \mathcal{H}_k,  k \in \N)$. The product of $\To$ is the usual composition of endomorphisms and the involution is the Hermitian adjoint. 
A very important fact is the existence of a natural star-algebra morphism $\sip : \To \rightarrow \Ci  ( M, \op{End} A) $ which is onto and with kernel 
$$ k^{-1} \To := \{ (k^{-1} T_k) / (T_k) \in \To \} = \{ (T_k) \in \To / \| T_k \| = \bigo (k^{-1}) \}.$$    
For any $T \in \To$, $\sip (T) $ is called the principal symbol of $T$. 
Furthermore, if $T$ and $S$ are two Toeplitz operators with scalar valued principal symbol $f$ and $g$ respectively, then $ik [T, S]$ is a Toeplitz operator with principal symbol $\{ f, g\}$, where $\{ \cdot, \cdot \}$ is the Poisson bracket corresponding to $\om$. 

 Let $\Tosc$ be the subalgebra of $\To$ consisting of operators with a scalar valued principal symbol. We say that a linear map $\sis : \Tosc \rightarrow \Ci ( M , \op{End} A)$ is a {\em subprincipal symbol map} if it satisfies the following conditions: 
\begin{enumerate} [label=\roman{*}),ref=\roman{*})]
\item \label{item:1} for any $T \in \To$,  $\sis ( k^{-1}T) = \sip (T)$,
\item \label{item:2} for any $T \in \Tosc$, $ \sis ( T^* )= \sis (T) ^*$, 
 \item \label{item:3} for any $T, S \in \Tosc$, $\sis (TS) = \sip(T) \sis (S) + \sis (T) \sip (S) + \tfrac{1}{2i} \{ \sip (T) , \sip (S) \}$,
\end{enumerate}
With such a map $\sis$, we can control the Toeplitz operators with a scalar principal symbol up to $k^{-2} \To$. More precisely, by \ref{item:1} and \ref{item:2}, the map 
$$ \si := \sip + \hb \sis : \Tosc \rightarrow  \Ci (M) \oplus \hb \; \Ci ( M , \op{End} A)$$
is real, onto with kernel $k^{-2} \Tosc$.

As a first attempt to construct a subprincipal symbol map, we can consider the contravariant symbols. Let $\Pi_k$ be the orthogonal projector of $\Ci ( M, L^k \otimes A)$ onto $\mathcal{H}_k$. Recall that for any $f \in \Ci (M, \op{End} A)$, $( \Pi_k f, \; k \in \N)$ is a Toeplitz operator with principal symbol $f$. So by the properties of the principal symbol recalled above, for any Toeplitz operator $T \in \To$, we have $ T= \Pi f \mod k^{-1} \To$ with $f = \sip (T)$. We can now set $\sis^{\op{c}} (T) := \sip ( k ( T - \Pi f ) )$, so that we have 
\begin{gather}  \label{eq:secondordre} 
T = \Pi f + k^{-1} \Pi g \mod k^{-2} \To \qquad \text{ with } f = \sip(T) , \; g = \sis ^{\op{c}} (T) 
\end{gather}
This defines a map $\sis^{\op{c}}$ which satisfies \ref{item:1}, \ref{item:2} but not \ref{item:3}. For instance, in the K{\"a}hler case with $A$ the trivial line bundle, we proved in \cite{oim_op} that 
$ \sis^{\op{c}} (ST) = \sip(T) \sis ^{\op{c}}(S) + \sis ^{\op{c}}(T) \sip (S) + \hb B ( \si(S) ,\si(T)) $
where in complex coordinates $(z^i)$, 
$$B( f,g) = - \sum G^{jk} \partial_{z_j} f . \partial_{\con{z}_k } g \quad \text{ if } \quad \om = i \sum G_{j,k} dz_j \wedge d\con{z}_k .$$ 
As expected, the antisymmetric part of $B$ is $\frac{1}{2i}$ times the Poisson bracket. But the symmetric part does not vanish. Nevertheless, setting 
\begin{gather} \label{eq:contsub}
\sis (T) = \sis ^{\op {c} } (T) + \frac{\hb}{2}  \Delta  \sip (T) ,
\end{gather}
where $\Delta = \sum  G^{jk} \partial_{z_j} \partial_{\con{z}_k}$, we obtain a subprincipal symbol map, which was introduced in \cite{oim_lag} to state the Bohr-Sommerfeld conditions. Equation (\ref{eq:contsub}) may be viewed as a generalization to K{\"a}hler manifolds of the formula giving the subprincipal term of the Weyl symbol in terms of the anti-Wick symbol. Still in the K{\"a}hler case, this approach can be generalized to any holomorphic vector bundle $A$ by using the expression for $B$ obtained in \cite{MaMa2012}.  In the general symplectic case, we only know that $B$ is a bidifferential operator. Using some standard argument in deformation quantization, this is enough to deduce the existence of a subprincipal symbol map, cf. Proposition \ref{prop:affine}. But we do not have an explicit formula defining this symbol.

Actually, there is a direct way to define an explicit subprincipal symbol. The construction is easier in the case where the canonical bundle $K = \wedge ^{n,0} T^*M$ has a square root $(\delta, \varphi)$, what we assume from now on. Here $\delta$ is a Hermitian line bundle over $M$ and $\varphi$ is an isomorphism from $\delta^2$ to $K$. It is known that such a square root exists if and only if $M$ has a spin structure if and only if the second Stiefel-Withney class of $M$ vanishes. 

 Set $B := A \otimes \delta^{-1}$ so that $A = B \otimes \delta$. 
Choose a Hermitian connection $\nabla^B$ of $B$.  Denote by $\nabla^k$ the connection of $L^k \otimes B$ induced by $\nabla $ and $\nabla^B$.
 For any vector field $X$ of $M$, let $D_X^K$ be the derivative of $\Ci (M, K)$ in the direction of $X$ given by $$ D_X ^K ( s) = p (\mathcal{L}_X s) , \qquad \forall s \in \Ci ( M , K),$$
where $p$ is the projection from $\wedge ^n (T^*M \otimes \C)$ onto $K$ with kernel $\wedge^{n-1, 1} T^*M \oplus \ldots \oplus \wedge^{0,n} T^*M$. Introduce now the derivative $D_X^\delta$ of $\Ci(M, \delta)$ in the direction of $X$ such that 
$$ D_X ^K ( \varphi(s^2) ) = 2 \varphi ( s \otimes D_X^\delta s ) , \qquad \forall s \in \Ci ( M, \delta). $$
Finally for any $f \in \Ci (M)$, consider the operator 
$$ Q_k (f) = \Pi_k \bigl( f + \tfrac{i}{k} \bigr( \nabla^{k}_X \otimes \op{id} + \op{id} \otimes D^\delta_X)\bigr) : \mathcal{H}_k \rightarrow \mathcal{H}_k .$$ 
where $X$ is the Hamiltonian vector field of $f$, that is $df = \om ( X, \cdot) $.

\begin{theo} \label{theo:intro_subpr}
For any Hermitian connection $\nabla^B$ of $B$, we have 
\begin{enumerate} 
\item \label{item:1theoint} For any $f \in \Ci (M)$, $Q(f) = \bigl(Q_k(f)\bigr)$ is a Toeplitz operator with principal symbol $f$. 

\item \label{item:2theoint} For any $f $ and $g \in \Ci (M)$, we have
$$ Q_k ( f) Q_k ( g) = Q_k ( fg)  + \tfrac{1}{2ik} Q_k ( \{ f, g \} ) \mod k^{-2} \To $$
and 
\begin{gather} \label{eq:com1}
 [ Q_k ( f) , Q_k ( g) ] = \tfrac{1}{ik} Q_k ( \{ f, g \} ) + \tfrac{1}{i k^2} \Pi_k( R(X,Y)) \mod k^{-3} \To , 
\end{gather}
where $X$ and $Y$ are the Hamiltonian vector fields of $f$ and $g$ respectively, $iR = \Theta (B) \in \Om^2 ( M, \op{End} B)$ is the curvature of $\nabla^B$. 

\item \label{item:3theoint} For any $f \in \Ci (M)$ and Toeplitz operator $S$, $(ki [ Q_k (f) , S ] )$ is a Toeplitz operator with principal symbol 
\begin{gather} \label{eq:com2} 
 \sip ( ki [ Q_k (f) , S ] ) = - \nabla_X^B (\sip (S)).
\end{gather} 
where $X$ is the Hamiltonian vector field of $f$. 
\end{enumerate}
\end{theo}

In this statement, we have identified $\op{End} A$ and $\op{End} B$ by using that $\delta$ is a line bundle so that $\op{End }(\delta)$ is canonically isomorphic with the trivial line bundle. Furthermore, we let the covariant derivative $\nabla^B$ act on sections of $\op{End} B$ in the usual way. Observe that in the case where $B$ is a line bundle  $\nabla_X^B (\sip (S)) = X. \sip (S) $ so that Equation (\ref{eq:com2}) becomes $ \sip ( ki [ Q_k (f) , S ] ) = \{ f, \sip (S) \} $. 

We can now defined a subprincipal symbol map  $\sis : \Tosc \rightarrow \Ci ( M , \op{End} A)$ by $\sis (T) := \sip ( k ( T - Q ( \sip (T))))$. In other words, for any Toeplitz operator $T \in \Tosc$, we have 
\begin{gather} \label{eq:sissis}
  T = Q(f) + k^{-1} \Pi g \mod k^{-2} \To, \qquad \text{ with } f = \sip(T) , \; g = \sis  (T) 
\end{gather}
By Theorem \ref{theo:intro_subpr}, $\sis$ satisfies \ref{item:1}, \ref{item:2} and \ref{item:3}. Furthermore, by (\ref{eq:com1}) and (\ref{eq:com2}), for any $S, T \in \Tosc$, 
\begin{gather} \label{eq:siscom}
 \sis ( ik [ T,S ] ) = R(X,Y) - \nabla^B _X \sis (S) + \nabla^B_Y \sis (T) + i [ \sis (T) , \sis (S) ]
\end{gather}
where $X$ and $Y$ are the Hamiltonian vector fields of $\sip (T)$ and $\sip (S)$ respectively. 

Observe that $\sis$ is not uniquely defined, since it depends on the choice of $\nabla^B$. Actually, the space of Hermitian connections of $B$ is an affine space directed by $\Om^1 ( M , \herm B)$ where $\herm B$ is the bundle of Hermitian endomorphisms of $B$. As we will see, the space of subprincipal symbol maps is an affine space directed by $\Ci ( M , TM \otimes \herm A)$, the action being given by $ (V+ \sis )(T) = \sis(T) + V. \sip (T)$.

\begin{prop} \label{prop:tous}
The map from the space of connections of $B$ to the space of subprincipal symbol maps of $\To$, sending $\nabla^B$ to $\sis$ so that (\ref{eq:sissis}) is satisfied, is an isomorphism of affine vector spaces. Here the corresponding vector spaces $\Om^1 ( M , \herm B)$ and $\Ci ( M , TM \otimes \herm A)$ are identified through the symplectic duality $T^*M \simeq TM$ and the natural isomorphism $\herm A \simeq \herm B \otimes \herm \delta \simeq \herm B$ considered above. 
\end{prop}

\begin{rem} 
The case where $B$ is trivial, so that $\mathcal{H}_k \subset \Ci ( M ,L^k \otimes \delta)$, is usually called the quantization with metaplectic correction. There is a natural and very particular choice for the connection of $B$: $\nabla^B = d$. Then (\ref{eq:com1}) simplifies  $$ [  Q_k ( f) , Q_k ( g) ] = \tfrac{1}{ik} Q_k ( \{ f, g \} )  \mod k^{-3} \To $$ 
and the corresponding subprincipal symbol satisfies: 
\begin{gather} \label{eq:siscomsimp}
\sis ( ik [ T,S ] ) = \{ \sip ( T) , \sis (S) \} + \{ \sis (T) , \sip (S) \}.
\end{gather} 
This situation can be compared with the one of pseudo-differential operators acting on half-densities. In that case, there is a well-defined subprincipal symbol which satisfies (\ref{eq:siscomsimp}). Let us note also that according to \cite{Bo2}, only a map satisfying Equation (\ref{eq:siscomsimp}) deserves the name of subprincipal symbol. 
\qed \end{rem}

\begin{rem} \label{rem:clas}
Equation (\ref{eq:com1}) is relevant to classify the algebras $\To$ corresponding to various choices of $A$. For instance, assume that $A$ is a line bundle so that $R \in \Om^2( M , \R)$. Then the cohomology class of $R$ does not depend on the choice of the subprincipal symbol. Actually, $i R$ being the curvature of $\nabla^B$, $$ \tfrac{1}{2\pi} [ R] = c_1 (B) = c_1 (A) + \tfrac{1}{2} c_1 ( M).$$ Assume now that we have another line bundle $A'$ on $M$ with corresponding spaces $\mathcal{H}_k' \subset \Ci ( M , L^k \otimes A')$ and Toeplitz algebra $\To'$. We may ask whether there exists a star-algebra morphism $ \Phi: \To \rightarrow \To'$  such that $\Phi ( k^{-1} T) =  k^{-1} \Phi (T)$ and $\sip  ( \Phi ( T) ) = \sip (T)$ for any $T \in \To$. Observe that if $\Phi$ is such a morphism and $\sis'$ is a subprincipal symbol map of $\To'$, then $\sis  := \sis' \circ   \Phi $ is a subprincipal symbol map of $\To$.  So as a consequence of Equation (\ref{eq:com1}), a necessary condition for $\Phi$ to exist is that $A$ and $A'$ have the same Chern class. In \cite{oim_eq}, we prove that this condition is also sufficient in the K{\"a}hler case. We plan to extend these results to the general symplectic case in a next paper. 
\qed \end{rem}

\begin{rem} \label{rem:Fedosov} 
Recall that the star products of $(M, \om)$ are classified up to equivalence by their Fedosov class which is an element of $H^2 (M , \C) [[\hb]]$, cf. \cite{Fe}. In the case where $A$ is a line bundle, the product $\star _{\op{cont}}$ of contravariant symbols of Toeplitz operators is a star product. By Equation (\ref{eq:com1}), the first coefficient of the Fedosov class of $\star_{\op{cont}}$ is equal to $  c_1 (A) + \tfrac{1}{2} c_1 ( M)$.  
\qed \end{rem}

\begin{rem} 
In the K{\"a}hler case, there is a particular choice for the connection of $B$, namely the Chern connection. Doing this choice and assuming that $A$ is a line bundle, we recover the subprincipal symbol defined in equation (\ref{eq:contsub}), cf. \cite{oim_demi}. In this case, Equations (\ref{eq:contsub}) and (\ref{eq:com1}) have been proved in \cite{oim_eq}. However, the proof in \cite{oim_eq} was rather indirect and based on the morphisms which we alluded to in Remark \ref{rem:clas}.  The proof we propose in this paper is much simpler.  
\qed \end{rem}

Theorem \ref{theo:intro_subpr} and Proposition \ref{prop:tous} generalize to the case where there is no  half-form bundle.  The difficulty as we will see is to define the convenient operators $Q_k( f)$ in this case and to understand what replaces the choice of the connection on $B$. 
 
The paper is organized as follows. In Section \ref{sec:toeplitz-operators}, we recall some basic facts on Toeplitz operators. In Section \ref{sec:subprincipal-symbols}, we start the study of subprincipal symbols and we go as far as possible without using the operators $Q_k (f)$. We will see that we can deduce Equation (\ref{eq:siscom}) from almost nothing, but without computing explicitly $R$ and $\nabla^B$.  In section \ref{sec:half-form-comp}, we introduce some material related to half-form bundles. In section \ref{sec:quantization-map}, we define the operator $Q_k(f)$ and state the theorem generalizing Theorem \ref{theo:intro_subpr} in the absence of half-form bundle. Sections \ref{sec:schw-kern-toepl} and \ref{sec:proof} are devoted to the proof of this result.

\section{Toeplitz operators} \label{sec:toeplitz-operators}

Consider a compact symplectic manifold $M$ equipped with a prequantum bundle $L \rightarrow M$. Recall that $L$ is a Hermitian line bundle with a connection of curvature $\frac{1}{i} \om$. Let $A $ be any Hermitian vector bundle. Let $j$ be an almost complex structure of $M$ compatible with $\om$. Consider a family 
$$\mathcal{H}= ( \mathcal{H}_k \subset \Ci ( M , A \otimes L^k) , \; k \in \N)$$ of finite dimensional subspaces. Assume that the orthogonal projector $\Pi_k$ onto $\mathcal{H}_k$ belongs to the algebra $\Fa_0$ introduced in Section \ref{sec:algebra-fa_0}. 

 A {\em Toeplitz operator}  is any family $(T_k : \mathcal{H}_k \rightarrow \mathcal{H}_k, \; k \in \N)$ of operators of the form 
\begin{gather} \label{eq:def_Toep}
 T_k = \Pi_k f( \cdot, k ) + R_k , \qquad k \in \N^*
\end{gather}
where $f(\cdot, k )$, viewed as a multiplication operator, is a sequence in $\Ci ( M, \op{End} A )$ admitting an asymptotic expansion $ f_0 + k^{-1} f_1 + \ldots $ for the $\Ci $ topology. Furthermore the norm of $R_k \in \op{End} \mathcal{H}_k$ is in $\bigo ( k^{-\infty}) = \bigcap_N \bigo ( k^{-N})$. 

The following facts are proved in \cite{oim_symp_quant}.  The space $\To$ of Toeplitz operators is a star-algebra with identity $(\Pi_k)$, the product being the usual composition of operators, the involution being the Hermitian adjoint. The symbol map $$\si_{\op{cont}} : \To \rightarrow  \Ci ( M , \op{End} A) [[\hb ]]$$ sending $(T_k)$ into the formal series $ f_0 + \hb f_1 + \ldots $ where the functions $f_i$ are the coefficients of the asymptotic expansion of the multiplicator $f( \cdot , k)$ is well defined. It is onto and its kernel is the ideal consisting of $\bigo ( k^{-\infty})$ Toeplitz operators. More precisely, for any integer $\ell$,
$ \| T_k \| = \bigo ( k^{-\ell})$ if and only if $\si_{\op{cont}} ( T) = \bigo ( \hb^{\ell}) .$ 
According to Berezin terminology, $\si_{\op{cont}} (T) $ is called the {\em contravariant} symbol of $T$. 

The {\em principal} symbol $\sip (T) \in \Ci ( M , \op{End} A)$ is by definition the first coefficient of the contravariant symbol, so $\si_{\op{cont}} (T) = \sip (T) + \bigo ( \hb)$. The principal symbol map $$\sip: \To \rightarrow \Ci (M, \op{End} A)$$ is onto and $\sip (T) = 0$ if and only if there exists $S \in \To$ such that $T = k^{-1} S$. For any Toeplitz operators $T, S \in  \To $ with principal symbols $f$ and $g$,  we have $ \sip ( T S) = f.g$. If $f$ and $g$ are scalar valued, then $ik [T,S] \in \To$ and 
\begin{gather}\label{eq:commutator}
 \sip   (ik [T,S] ) = \{ f, g \},
\end{gather}
where $\{ \cdot, \cdot \}$ is the Poisson bracket of $M$. 
Denoting by $\| T_k \|$ the operator norm of $T_k$ corresponding to the scalar product of $\mathcal{H}_k \subset \Ci ( M , L^k \otimes A)$, we have
$$  \| T_k \|  = \sup_{y \in M} | \sip (T) (y) | + \bigo ( k^{-1}) $$ 
where for any $y \in M$, $| \sip (T) (y) |$ is the operator norm of $\sip (T) (y) \in \op{End} A_y $. 
Consequently, $\sip(T) =0$ if and only if $\| T_k \| = \bigo ( k^{-1})$. 

As a last property, the full product of contravariant symbols has the following form: if  $\si_{\op{cont}} ( T) = \sum \hb^\ell f_\ell$ and $\si_{\op{cont}} ( S) = \sum \hb^\ell g_{\ell} $, then 
\begin{gather} \label{eq:prod_cont_complet} 
 \si_{\op{cont}} ( T S) = \sum_{\ell} \hb^{\ell} \sum_{\ell = p+q+ r} B_r ( f_p , g_q)  
\end{gather}
where $B_0 ( f,g) = fg$ and for any $r$, $B_r: \Ci ( M ,\op{End} A) \times \Ci ( M , \op{End} A) \rightarrow \Ci ( M, \op{End} A)$ is a bilinear local operator, cf. \cite{oim_symp_quant}.

\section{Subprincipal symbols} \label{sec:subprincipal-symbols}

Denote by $\Tosc$ the set of Toeplitz operator $T \in \To$ with a scalar principal symbol.
Let $\synorsp$ be the set of linear map $\sis : \Tosc \rightarrow \Ci ( M , \op{End} A)$ satisfying the following conditions: 
\begin{enumerate} [label=\roman{*}),ref=\roman{*})]
\item \label{item:def}  for any $T \in \To$,  $\sis ( k^{-1}T) = \sip (T)$,
\item \label{item:adj}  for any $T \in \Tosc$, $ \sis ( T^* )= \sis (T) ^*$, 
 \item \label{item:prod} for any $T, S \in \Tosc$, $\sis (TS) = \sip(T) \sis (S) + \sis (T) \sip (S) + \tfrac{1}{2i} \{ \sip (T) , \sip (S) \}$,
\end{enumerate}
We will first prove that $\synorsp$ is not empty.
Define the contravariant subprincipal symbol map $ \sis^{\op{c}} : \Tosc \rightarrow  \Ci ( M , \op{End} A) $ as follows:
\begin{gather} \label{eq:sis_cont}
\si_{\op{cont}} ( T) = \sip(T) + \hb \sis^{\op{c}} (T)  + \bigo (\hb^2)
\end{gather} 
Then $\sis^{\op{c}}$ satisfies \ref{item:def}, \ref{item:adj} but not \ref{item:prod}. Actually, by (\ref{eq:prod_cont_complet}), we have
\begin{gather} \label{eq:defB}
 \sis^{\op{c}} (TS) = \sip(T) \sis^{\op{c}} (S) + \sis^{\op{c}} (T) \sip (S) + B( \sip (T) , \sip (S) )
\end{gather}
where $B$ is a bidifferential operator from $\Ci ( M ) \times \Ci ( M )$ to $\Ci ( M, \op{End} A)$. By (\ref{eq:commutator}), the antisymmetric part of $B$ is equal to $\frac{1}{2i}$ times the Poisson bracket, so 
\begin{gather} \label{eq:Bsym}
 B(f,g) = B^{\op{s}} (f,g) + \tfrac{1}{2i} \{ f , g \} 
\end{gather}
where $B^{\op{s}}$ is symmetric. We will modify $\sis^{\op{c}}$ to get a subprincipal symbol satisfying \ref{item:prod}. The method we follow is based on standard lemmas in Deformation quantization. 

\begin{prop} \label{prop:Hoschild}
Let $B^{\op{s}} : \Ci ( M) \times \Ci ( M ) \rightarrow \Ci ( M , \op{End} A)$ be any bidifferential symmetric operator satisfying 
\begin{gather} \label{eq:3}
 B ^{\op{s}}( f, g) h + B^{\op{s}} ( fg, h) = B^{\op{s}}( f, gh) + f B^{\op{s}}( g,h) , \quad \forall f,g,h \in \Ci ( M ) .
\end{gather}
Then there exists a differential operator $Q : \Ci ( M) \rightarrow \Ci ( M, \op{End} A) $ such that for any $f, g \in \Ci (M)$, we have $B^{\op{s}}( f, g) = fQ (g) + Q(f) g - Q( fg)$. 
\end{prop} 

In the case where $A$ is the trivial line bundle over $M$, this result is a classical lemma in Hochschild cohomology saying that any cocycle with a null antisymmetric part is exact, cf. \cite{BeCaGu} for a short proof. It is easy to deduce Proposition \ref{prop:Hoschild} from this particular case. 
 
Using that the product of Toeplitz operators is associative, a straightforward computations shows that the bidifferential operator $B$ defined by (\ref{eq:defB}) satisfies the equality $ B( f, g) h + B ( fg, h) = B( f, gh) + f B( g,h)$. Using that the Poisson bracket is a derivative with respect to each of his arguments, we conclude that the symmetric part $B^{\op{s}}$ of $B$ satisfies (\ref{eq:3}). Applying Proposition \ref{prop:Hoschild}, we get a differential operator $Q$. Then a straightforward computation shows that  $\sis := \sis ^{\op{c}} + Q \circ \sip$ satisfies  \ref{item:prod}. Since $\sis ^{\op{c}}$ satisfies \ref{item:adj}, $B^{\op{s}}$ is real in the sense that $B^{\op{s}}$ is the complexification of a $\R$-bilinear map $\Ci ( M, \R) \times \Ci ( M , \R) \rightarrow \Ci ( M , \herm A)$, where $\herm A \rightarrow M$ is the vector bundle of Hermitian endomorphisms of $A$.
$B^{\op{s}}$ being real, we can choose $Q$ real, so that $\sis^{\op{c}}$ satisfies \ref{item:adj} as well. We have proved that $\synorsp$ is not empty. 

Now consider $\sis \in \synorsp$ and $V \in \Ci ( M , TM \otimes \herm A)$. Then the map $\sis'$ defined by 
 $$\sis' (T) := \sis (T) + df (V) \qquad \text{ where } \qquad f = \sip(T) ,$$
satisfies \ref{item:def}, \ref{item:adj}, \ref{item:prod}. Conversely, let $\sis$ and $\sis'$ in $\synorsp$. By  \ref{item:def}, $\sis' - \sis$ vanishes on $k^{-1} \To$. Since $k^{-1} \To$ is the kernel of the principal symbol map, we have $ \sis' (T) - \sis (T) = D( \si_p(T))$ where $D: \Ci ( M) \rightarrow \Ci ( M , \op{End}A)$. By \ref{item:adj}, $D$ is real. By \ref{item:prod}, $D$ is a derivation, that is $D(fg) = fD(g) + gD(f)$. We conclude that $D (f) = df (V)$ for some $V \in \Ci ( M , TM \otimes \herm A)$. To summarize we have proved the 
following Proposition.

\begin{prop} \label{prop:affine}
$\synorsp$ is an affine space with associated vector space $\Ci ( M , TM \otimes \herm A)$. 
\end{prop}

It is helpful to view $\sis \in \synorsp$ as a first order deformation of the principal symbol. To give a sense to this, denote by $\sy$ be the vector space $  \Ci (M) \oplus \hb \Ci ( M , \op{End}A)$
and define the map $\si :\Tosc \rightarrow \sy$  by $$\si (T) = \sip (T )  + \hb \sis (T).$$ Using that $\sip$ is onto with kernel $k^{-1} \To$ and that $\sis$ satisfies \ref{item:def}, we see that $\si$ is onto with kernel $k^{-2} \To$. By \ref{item:adj}, $\si$ is real. Furthermore by \ref{item:prod}, $\si (ST) = \si (S) \star \si (T)$ where $\star$ is the product of $\sy$ given by 
$$ (f_0 + \hb f_1 ) \star ( g_0 + \hb g_1) = f_0 g_0 + \hb ( f_0 g_1 + f_1 g_0 + \tfrac{1}{2i} \{ f_0 , g_0\} ).  $$
$1$ being the identity of $\star$, we easily get that $\si ( \op{id} ) = 1$. 

For any $T, S \in \Tosc$, $k[T,S] $ is a Toeplitz operator with scalar symbol. Furthermore observe that the class of $k [T,S]$ modulo $\bigo( k^{-2})$ only depend on the classes of $T$ and $S$ modulo $\bigo ( k^{-2})$. So  there exists a unique bilinear map $$[\cdot , \cdot ]_{\si} : \sy \times \sy \rightarrow \sy$$ such that $\si ( ik [T,S] ) = [\si (T), \si (S) ]_{\si}$ for any $T$, $S \in \Tosc$. Since the commutator of endomorphisms is a derivation with respect to each argument and satisfies the Jacobi identity, we obtain that  for any $f,g, h \in \sy$  
\begin{gather} \label{eq:derrho}
[f \star g , h ]_{\si} = f \star [ g,h]_{\si} + [ f , h ]_{\si} \star g.
\end{gather}
and
\begin{gather} \label{eq:jacobirho}
[ f, [g, h ]_{\si} ]_{\si} + [g, [ h, f]_{\si}]_{\si} + [ h, [f, g ]_{\si}]_{\si} = 0 
\end{gather}
Furthermore, $[f,g ]_{\si} = -[g,f]_{\si}$ and $[f ^* , g ^* ] _{\si} = [f,g]_{\si}^*$.  Exploiting these equations, we deduce the following proposition. 

\begin{prop} \label{prop:commutator}
For any $\sis \in \synorsp$, there exists $R \in \Om^2 ( M, \herm A)$ and a connection $\nabla : \Ci ( M , \herm A) \rightarrow \Om^1 (M, \herm A)$ such that for any $f_0 + \hb f_1$, $g_0 + \hb g_1 \in \sy$ we have
$$ [ f_0 + \hb f_1 , g_0 + \hb g_1 ] _\si = \{ f_0 , g_0 \} + \hb \bigl( R( X,Y) - \nabla_X g_1 + \nabla_Y f_1 + i [ f_1 , g_1 ] \bigr) $$
where $X$ and $Y$ are the Hamiltonian vector fields of $f_0$ and $g_0$. Furthermore, 
$$ \op{courb} \nabla = i \op{ad}_R, \qquad \nabla R = 0 , \qquad \nabla \op{id} = 0 $$and $ \nabla (f_1 .g_1) = (\nabla f_1) . g_1 + f_1 .( \nabla g_1)$
for any $f_1, g_1 \in \Ci ( M , \herm A)$. 
\end{prop} 
\begin{proof} 
By condition \ref{item:def} and the properties of the principal symbol, we have 
$$ [ f_0 + \hb f_1 , g_0 + \hb g_1 ] _\si = \{ f_0 , g_0 \} + \hb \bigl( A( f_0 , g_0) + B( f_0 , g_1) +  C( g_0 , f_1)  + i [ f_1 , g_1 ] \bigr) ,$$
where $A$, $B$, $C$ are bilinear operators with value in $\Ci ( M , \op{End} A)$, $A$ being defined on $ \Ci (M) \times \Ci ( M)$ and $B,C$ on $ \Ci ( M) \times \Ci ( M , \op{End} A)$.  Since $[\cdot, \cdot]_{\si}$ is antisymmetric, we have that\begin{gather} \label{eq:antisym} 
 A( f_0 , g_0 ) = - A( g_0 , f_0) , \qquad B( f_0 , g_1) = - C( f_0, g_1)
\end{gather} 
Since $\si$ is real and $\si ( \op{id} ) =1$, $[ \cdot, \cdot]_{\si}$ is real and $[ 1, \cdot ]_{\si} = 0$. Consequently, $A$ and $B$ are real, meaning that $A( \con{f}_0 , \con{g}_0 ) = A( f_0 , g_0 ) ^*$ and  $B ( \con{f}_0, g_1^* ) = B( f_0 , g_1) ^*$. Furthermore 
\begin{gather}  \label{eq:const} 
 A( 1, g_0 ) = 0, \qquad B( 1, g_1 ) = 0 , \qquad B( f_0 , \op{id} ) = 0 .
\end{gather}  
The last equation follows from the fact that $ \si ( k^{-1} \op{id} ) = \hb \op{id}$ so that $[ \hb \op{id} , \cdot ]_{\si} = 0$. 
By equation (\ref{eq:derrho}), for any $f_0$, $g_0 \in \Ci ( M )$, and  $  h_1 \in \Ci (M , \op{End} A)$
$$ [ f_ 0 \star g_0 , \hb h_1 ]_{\si} =   f_0 \star [ g_0 , \hb h_1 ]_{\si}  +  [ f_0 , \hb h_1 ]_{\si} \star g_0  $$
so that 
\begin{gather} \label{eq:2}
 B( f_0g_0 , h_1) = f_0 B ( g_0 , h_1) + g_0 B (f_0 , h_1) 
\end{gather} 
Furthermore, as another application of Equation (\ref{eq:derrho}),
$$ [ f_ 0 ,   g_0  \star \hb h_1 ]_{\si} =   [ f_0,  g_0  ]_{\si} \star  \hb h_1  +  g_ 0 \star [ f_0 , \hb h_1 ]_{\si}     $$
so that 
\begin{gather} \label{eq:1}
 B( f_0 , g_0 h_1 ) = \{ f_0 , g_0 \} h_1 + g_0 B( f_0 , h_1) = -(X.g_0) h_1 + g_0 B( f_0 , h_1).
\end{gather} 
where $X$ is the Hamiltonian vector field of $f_0$. By Equation (\ref{eq:1}), $B ( f_0, \cdot)$ is a derivative of $\Ci ( M , \op{End} A)$ in the direction of $-X$. By Equation  (\ref{eq:2}) and the second equation of (\ref{eq:const}), for any $p \in M$,  $B (f_0, g_1) (p) = B (f_0', g_1 )(p)$ if $f_0$ and $f_0'$ have the same differential at $p$. These two facts imply that $B( f_0, h_1) = -\nabla_X h_1$ for a connection $\nabla : \Ci (M, \op{End} A) \rightarrow \Om^1 ( M, \op{End} A)$. Since $B$ is real,  $\nabla$ is actually the complexification of a connection $ \Ci ( M , \herm A) \rightarrow \Om^1 (M, \herm A)$. Furthermore by the last equation of (\ref{eq:const}), $\nabla \op{id} = 0 $. 

Consider now $f,g,h \in \Ci (M)$. Expanding Equation (\ref{eq:derrho}) and using Jacobi identity for the Poisson bracket, we obtain that 
$$ A( fg, h ) = f A( g, h) + g A( f,h).$$
Since furthermore $A$ is antisymmetric and vanishes on the constant, we conclude that $A(f, g) = R ( X, Y)$ with $R \in \Om^2 (M, \op{End} A)$ and $X$, $Y$ the Hamiltonian vector fields of $f$ and $g$. Since $A$ is real, $R \in \Om^{2 }(M, \herm A)$.  

Still working with $f,g,h \in \Ci (M)$ and denoting by $X$, $Y$ and $Z$ their Hamiltonian vector fields, we have
\begin{xalignat*}{1}
 [ f , [g, h ]_{\si} ]_{\si} = & \{ f, \{ g , h \} \} + \hb \bigl( A( f, \{ g , h \} ) +  B( f , A(g,h) ) \bigr) \\
= & \{ f, \{ g , h \} \} - \hb \bigl( R( X, [Y,Z]) + \nabla_X R( Y,Z) \bigr)
\end{xalignat*}
Since the bracket $[\cdot, \cdot]_{\si}$ satisfies the Jacobi identity, this shows that $(\nabla R )(X,Y, Z) = 0 $ so that $\nabla R=0$.  
Another application of the Jacobi identity (\ref{eq:jacobirho}) with $f, g \in \Ci(M)$ and $h \in \hb \Ci ( M , \op{End} A)$ leads to
$$ [ \nabla_X , \nabla_Y ] h_1 - \nabla_{[X,Y] } h_1  = i [ R( X,Y) , h_1 ] $$
which means that the curvature of $\nabla$ is $i\op{ad}_R$. 
The last equation to prove does not follows from the previous conditions. Consider three Toeplitz operators $R$, $S$ and $T$ with principal symbols $f_0 \in \Ci(M)$, $g_1 , h_1 \in \Ci ( M , \op{End} A)$. Starting from $ik [ T , k^{-1}RS ]=  R (ik [T, k^{-1} S]) + (ik [ T, k^{-1} R]) S$, we obtain that   $$ \nabla_X (g_1 .h_1) =g_1 .( \nabla_X h_1) + (\nabla_X g_1) . h_1 $$ where $X$ is the Hamiltonian vector field of $f_0$. 
\end{proof} 

We can also easily compute how the form $R$ and the connection $\nabla$ are changed when we choose a new subprincipal symbol. 

\begin{lemme} \label{lem:chang_jauge}
Let $\sis \in \synorsp$, $V \in \Ci ( M,TM \otimes  \herm A)$ and set $\sis ' =   \sis + V .\sip$. Then the connections $\nabla$, $\nabla '$ and the two forms $R$, $R'$ corresponding respectively to $\sis $ and $\sis'$ satisfy
$$ \nabla' = \nabla + i \op{ad}_{\al}, \qquad R' = R + \nabla \al + i \al \wedge \al $$
where $\al  \in \Om^1 ( M, \herm A)$ is given by $\al = \om (V, \cdot)$. 
\end{lemme} 

The case where $A$ is a line bundle is already interesting. By the following corollary, the cohomology class $[R ] \in H^2(M, \R)$ does not depend on the choice of the subprincipal symbol map.  
\begin{cor} \label{cor:Fedosov_class}
Assume that $A$ is a line bundle. Then for any $\sis \in \synorsp$, we have 
$$[ f_0 + \hb f_1 , g_0 + \hb g_1 ]_{\si}  = \{ f_0 , g_0 \} + \hb \bigl( R( X,Y) - \mathcal{L}_X g_1 + \mathcal{L}_Y f_1  \bigr) $$
where $R \in \Om^2 ( M , \R)$ is closed and $\mathcal{L}_X$, $\mathcal{L}_Y$ are the Lie derivatives with respect to the Hamiltonian vector fields of $f_0$ and  $g_0$ respectively. Furthermore, the two-form $R'$ corresponding to $\sis ' =  \sis + V. \sip$ is given by $R' = R + d \iota_V \om $. So the cohomology class of $R$ does not depend on the choice of $\sis$.  
\end{cor} 

\begin{proof}
Since $A$ is a line bundle, $\herm A$ is naturally isomorphic with the trivial real line bundle. The fact that $\nabla \op{id} =0 $ implies that $\nabla$ is the de Rham derivation.
\end{proof} 

At this point, we could believe that for any $\sis \in \synorsp$, there exists some connection $\nabla^A : \Ci ( M , A) \rightarrow \Om^1 (M , A)$ preserving the metric of $A$, with curvature $R$ and such that $\nabla$ is the corresponding connection of $\herm A$. Indeed this would explain the equations in Proposition \ref{prop:commutator}. We could also believe that the connection corresponding to $\sis ' =  \sis + V. \sip$ is given by $\nabla^A + \frac{1}{i} \iota_V \om$ which would imply the equations in  Lemma \ref{lem:chang_jauge}. 

As we will see, this explanation almost holds, but we have to take into account the metaplectic correction. For instance, we will see that in the case where $A$ is a line bundle, the cohomology class of $[\frac{1}{2\pi} R]$ is not $c_1 (A)$ but $c_1 ( A)+ \frac{1}{2} c_1 (M)$.

\section{Half-form computations}  \label{sec:half-form-comp}
\subsection{Canonical bundle and derivatives} 
Let $2n$ be the dimension of $M$. 
Let $K = \wedge ^{n,0} T^*M$ be the canonical bundle of $M$ with respect to $j$. Denote by $p$ the projection $ \wedge ^{n} (T^*M \otimes \C) \rightarrow \wedge ^{n,0} T^*M$ with kernel $\wedge ^{n-1,1} T^*M \oplus \ldots \oplus \wedge ^{0,n} T^*M$. Then for any vector field $X$ of $M$, introduce the derivative 
\begin{gather} \label{eq:dercan}
 D^K_X : \Ci ( M , K) \rightarrow \Ci ( M , K), \qquad D_X^K \mu = p \mathcal{L}_X \mu
\end{gather}
where $ \mathcal{L}_X$ is the Lie derivative with respect to $X$. 

\begin{lemme} \label{lem:der}
For any $X \in \Ci ( M , TM)$ and $f \in \Ci ( M)$ we have
$$ D_{fX}^K = f D_X^K +  df ( X^{1,0} ) ,$$
where $X^{1,0} \in \Ci ( M, T^{1,0} M)$ is the $(1,0)$-component of $X$. 
Furthermore if $X$ is symplectic, $D_X^K$ preserves the metric of $K$ induced by $\om$. 
\end{lemme} 

$D_X^K$ preserves the metric means that for any sections $s,t \in \Ci (M, K)$ we have
$$ \mathcal{L}_X (s,t)= ( D_X^K s , t) + ( s ,D_X^K t) $$
where $(s,t)$ is the pointwise scalar product. Equivalently we say that $D_X^K$ is Hermitian. 

\begin{proof} 
By Cartan formula, we have $\mathcal{L}_{fX} - f\mathcal{L}_X = df \wedge \iota_X $. Introduce a local frame $(\partial_i)$ of $T^{1,0} M$ and denote by $(\te_i)$ the dual frame. Since $D_{f X}^K $ and $f D_X^K$ are both derivatives  in the direction of $fX$, they differ by the multiplication by a function. We compute this function by testing on  the frame $\te = \te_1 \wedge \ldots \wedge \te_n$ of the canonical bundle. Write $X^{1,0} = \sum X_i \partial_i $ so that 
$$ \iota_X \te = \sum ( -1) ^{i+1} X_i \te_1 \wedge \ldots \wedge \hat{\te}_i \wedge \ldots \wedge \te_n$$ 
and 
$$ p( df \wedge \iota_X \te) =  \sum X_i df ( \partial_i) \te  = df ( X^{1,0} ) \te .$$
Consequently, $ D_{fX}^K = f D_X^K +  df ( X^{1,0} )$. 

Let us prove that $D_X^K$ preserves the metric of $K$ when $X$ is symplectic. Recall that for any sections $s, t$ of $K$, $ s \wedge \con t =  C_n (s, t) \om ^n $
for some constant $C_n$ independent of $s$, $t$. Since 
$$ \mathcal{L}_X ( s \wedge \con t ) = (\mathcal{L}_X s ) \wedge \con{t} + s \wedge \con{ \mathcal{L}_X   t  } = (D_X^K s ) \wedge \con{t} + s \wedge \con{ D_X^K t}, $$
we deduce that  $\mathcal{L}_X (s,t)= ( D_X^K s , t) + ( s ,D_X^K t)$  when $X$ satisfies $\mathcal{L}_X \om^n = 0 $. 
\end{proof} 

\begin{lemme} \label{lem:com}
For any vector fields $X,Y$ of $M$, we have
$$  [D^K_X , D^K_Y ] = D^K _{[X,Y]} + B_j(X,Y)$$
where $B_j(X,Y)$ is the function of $M$ given by
$$ B_j (X,Y) = \sum  (\mathcal{L}_X \te_i)( \con{\partial}_j)  (\mathcal{L}_Y \con{\te}_j )( \partial_i) -  (\mathcal{L}_Y \te_i)( \con{\partial}_j)  (\mathcal{L}_X \con{\te}_j )( \partial_i)
$$
with $(\partial_i)$ a local frame of $T^{1,0}M$ and $( \te_i)$ the dual frame. 
\end{lemme}

\begin{proof} 
Since 
$$ p \mathcal{L}_X p \mathcal{L}_Y p - p \mathcal{L} _X \mathcal{L}_Y p =  p \mathcal{L}_X ( p - \op{id} ) \mathcal{L}_Y p,$$ 
we have
$$ [D^K_X , D^K_Y ] - D^K _{[X,Y]} = p \mathcal{L}_X ( p - \op{id} ) \mathcal{L}_Y  - p  \mathcal{L}_Y ( p - \op{id} ) \mathcal{L}_X    $$
Let $\te = \te_1 \wedge \ldots \wedge \te_n$. We have 
$$ \mathcal{L}_Y \te = \sum ( -1) ^{i+1} (\mathcal{L}_Y \te_i) \wedge \te_1 \wedge \ldots \wedge \hat{\te}_i \wedge \ldots \wedge \te_n $$
 so that
$$ (p - \op{id} )  \mathcal{L}_Y \te  = \sum ( -1) ^{i} (\mathcal{L}_Y \te_i)( \con{\partial}_j) \con{\te}_j   \wedge \te_1 \wedge \ldots \wedge \hat{\te}_i \wedge \ldots \wedge \te_n $$
and
$$ p \mathcal{L}_X ( p - \op{id} ) \mathcal{L}_Y \te = - \sum  (\mathcal{L}_Y \te_i)( \con{\partial}_j)  (\mathcal{L}_X \con{\te}_j )( \partial_i) \te $$
The final result follows.  
\end{proof}

\subsection{Metaplectic correction} \label{sec:metapl-corr}

Consider the set $\der$ of linear map $D : \Ci ( M , TM) \rightarrow \op{End} ( \Ci ( M , A))$ satisfying the following conditions
\begin{gather}
\label{eq:der1} D_X ( fs) = (X.f) s + f D_X s \\
\label{eq:der2} D_{fX} s = f D_X s + \tfrac{1}{2} df( X^{1,0}) s 
\end{gather}
for any vector field $Y$ of $M$, function $f \in \Ci ( M)$ and sections $s, t \in \Ci ( M , A)$. 
In the case where $X$ is a symplectic vector field, we also require that $D_X$ is Hermitian, 
 \begin{gather}
\label{eq:der3} X. ( s, t) = ( D_X s , t ) + ( s, D_X t)  , \qquad s, t \in \Ci ( M , A) 
\end{gather}
 where  $(s,t) \in \Ci ( M )$ is the pointwise scalar product of $s$ and $t$.
Observe that for any $D \in \der$ and $\al \in \Om^1 ( M , \herm A)$, $D + \frac{1}{i} \al$ belongs to $\der$. 

\begin{prop} \label{prop:der}
The space $\der$ is a real affine space directed by $\Om^1 ( M , \herm A)$.   
\end{prop}

\begin{proof}
The only difficulty is to check that $\der $ is not empty. To do that, we will use the derivations $D_X^K$ introduced in (\ref{eq:dercan}).  
Let $\nabla^A$ and $\nabla ^K $ be Hermitian connections of $A$ and $K$ respectively. For any vector field $X$, define
$$ D_X := \nabla^A_X + \tfrac{1}{2} B( X) \op{id} _A : \Ci ( M , A) \rightarrow \Ci ( M , A) $$
where $ B(X) \in \Ci ( M)$ is given by $B( X) = D^K_X - \nabla ^K _X$. $D_X $ is clearly a derivation in the direction of $X$. By Lemma \ref{lem:der}, $D_X $ satisfies Condition (\ref{eq:der2}) and preserves the metric when $X$ is symplectic. 
\end{proof}

\begin{rem} \label{rem:demi}
Let $\delta$ be a half-form bundle, that is a Hermitian line bundle such that $\delta^2$ is isomorphic to $K$. For any vector field $X$ of $M$, let $D_X^\delta$ be the derivative of $\Ci ( M , \delta)$ in the direction of $X$ such that
$$  D_X^K s^2 = 2 s \otimes D_X^\delta s , \qquad \forall s \in \Ci ( M , \delta). $$  Write $A = B \otimes \delta$ where $B = A \otimes \delta^{-1}$. Then there is a one to one correspondence between the space of Hermitian connections of $B$ and $\der$ given as follows: for any Hermitian connection $\nabla^B$ of $B$, we set
$$ D_Y = \nabla^B_Y \otimes \op{id} + \op{id} \otimes D_Y^{\delta} .$$
This provides another proof of Proposition \ref{lem:der} in the case where $M$ has a half-form bundle. 
\qed \end{rem} 

\begin{rem} 
Assume that $A$ is a line bundle. Then there is a one to one correspondence between the space of Hermitian connections of $C= A^2 \otimes K^{-1}$ and $\der$ given as follows: for any Hermitian connection $\nabla^C$ of $C$, we define
first 
$$ D^{A^2}_X = \nabla^C_X \otimes \op{id} + \op{id} \otimes D^K_X $$
and then $D_X $ is the unique derivation of $\Ci ( M , A)$  with respect to $X$ satisfying 
$$ D_X^{A^2} ( s^2)   = 2 s \otimes D_X s , \qquad \forall s \in \Ci ( M , A).\qed $$
 \end{rem} 

In the following proposition, we compute some kind of curvature for $D \in \der$.
\begin{prop} \label{prop:met_com}
For any $D \in \der $, there exists $R \in \Om^2 ( M , \herm A)$ and a covariant derivation $\nabla  : \Ci ( M , \herm A) \rightarrow \Om^1 ( M , \herm A)$ such that for any vector fields $X,Y$ of $M$
\begin{gather} \label{eq:fausse_Courbure}
 [D_X , D_Y ] = D_{[X,Y]} + i  R(X,Y) + \tfrac{1}{2} B_j(X,Y)
\end{gather}
with $B_j(X,Y)$ the function defined in Lemma \ref{lem:com} and for any $f \in \Ci ( M , \herm A)$ and $s \in \Ci ( M, A)$, 
\begin{gather} \label{eq:conend}
 D_X ( f.s) = (\nabla_X f).s + f .D_X . 
\end{gather}
Furthermore, the curvature of $\nabla$ is $i \op{ad}_R$, $\nabla R = 0 $, $\nabla \op{id} = 0 $ and for any  $f, g \in \Ci ( M , \op{End}  A)$, $\nabla (f. g) = (\nabla f). g  + f.( \nabla g)$.  
\end{prop}

\begin{proof} 
Assume first as in Remark \ref{rem:demi} that $A = B \otimes \delta$ and $ D_Y = \nabla^B_Y \otimes \op{id} + \op{id} \otimes D_Y^{\delta}$. Then we have a natural identification $\herm A \simeq \herm B$. We set $R = \frac{1}{i} \op{courb} \nabla^B$. Equation (\ref{eq:fausse_Courbure}) follows from Lemma \ref{lem:com}.  We define $\nabla $ as the connection of $\herm B$ induced by $\nabla^B$, so that Equation (\ref{eq:conend}) is satisfied and the properties of $\nabla$  given in the proposition are standard properties.

We can now extend to the result to the case where there is no half-form bundle. Observe first that $R$ and $\nabla$ are uniquely determined by Equations (\ref{eq:fausse_Courbure}) and (\ref{eq:conend}). Since this unicity is also local, the local existence of $R$ and $\nabla$ implies their global existence. But each point of $M$ has a neighborhood admitting a half-form bundle, where we can apply the first part of the proof. 
\end{proof}

\section{The quantization map}  \label{sec:quantization-map}

Consider $D \in \der$. For any $f \in \Ci ( M)$, define the derivative $P_{f,k}$ in the direction of the Hamiltonian vector field $X$ of $f$ 
$$ P_{f,k} = ( \nabla^{L^k}_X \otimes \op{id} + \op{id} \otimes D_X) : \Ci ( M , L^k \otimes A) \rightarrow \Ci ( M , L^k \otimes A ) . $$
Then we set 
$$ Q_k^D(f ) = \Pi_k \bigl( f + \tfrac{i}{k} P_{f,k}  \bigr) : \mathcal{H}_k \rightarrow \mathcal{H}_k $$

\begin{theo} \label{theo:quant}
Let $D \in \der$.
\begin{enumerate}
\item  For any $f \in \Ci (M)$, $\bigl(Q^D_k(f)\bigr)$ is a Toeplitz operator with principal symbol $f$. 
\item For any $f $ and $g \in \Ci (M)$, we have
$$ Q^D_k ( f) Q^D_k ( g) = Q^D_k ( fg)  + \tfrac{1}{2ik} \Pi_k  \{ f, g \}  + \bigo ( k^{-2}) $$
and 
$$ ik [ Q^D_k ( f) , Q^D_k ( g) ] = Q^D_k ( \{ f, g \} ) + \tfrac{1}{k} \Pi_k R (X,Y) + \bigo(k^{-2}), $$
where $X$ and $Y$ are the Hamiltonian vector field of $f$ and $g$ respectively and $R \in  \Om^2 ( M, \herm A)$ is defined as in Proposition \ref{prop:met_com}. 
\item For any $f \in \Ci (M)$ and Toeplitz operator $S$, $(ki [ Q_k^D (f) , S ] )$ is a Toeplitz operator with principal symbol 
$$ \sip ( ki [ Q_k^D (f) , S ] ) = - \nabla_X (\sip (S)).$$ 
where $\nabla $ is the connection of $\herm A$ defined in Proposition \ref{prop:met_com} and $X$ is the Hamiltonian vector field of $f$.  
\end{enumerate}
\end{theo} 

The theorem will be proved in Section \ref{sec:proof}.  
We can now defined a subprincipal symbol map  $\sis^D : \Tosc \rightarrow \Ci ( M , \op{End} A)$ by 
\begin{gather} \label{eq:defsis}
\sis ^D (T) := \sip ( k ( T - Q^D ( \sip (T)))), \qquad \forall T \in \Tosc
\end{gather}
so that we have 
\begin{gather} \label{eq:4}  
T = Q^D(f) + k^{-1} \Pi g \text{ modulo } k^{-2} \To \quad  \text{ with } f = \sip(T) \text{ and } g = \sis^D  (T).
\end{gather}


\begin{cor} 
For any $D \in \der$, $\sis^D$ belongs to $\synorsp$. The two form $R$ and the connection $\nabla$ corresponding to $\sis^D$, cf. Proposition \ref{prop:commutator}, are the ones corresponding to $D$, cf. Proposition \ref{prop:met_com}.  Furthermore the map sending $D$ into $\sis^D$ is an affine space isomorphism from $\der$ to $\synorsp$. 
\end{cor} 

\begin{proof}  
$\sis^D$ clearly satisfies Conditions \ref{item:def} and \ref{item:adj}. Condition \ref{item:prod} follows from the first equation of the second assertion of Theorem \ref{theo:quant}.
The fact that the two form $R$ and the covariant derivative $\nabla$ corresponding to $D$ and $\sis^D$ are the same follows from the second and third assertion of Theorem \ref{theo:quant}. Finally, recall that both $\der $ and $\synorsp$ are affine spaces directed respectively by $\Om^1 (M, \herm A)$ and $\Ci (M, TM \otimes \herm A)$, cf. Propositions \ref{prop:affine}, \ref{prop:der}. These vector spaces are isomorphic through $\om$. The map sending $D$ to $\sis^D$ is a morphism of affine spaces. 
\end{proof}

\section{Schwartz kernel of Toeplitz operators} \label{sec:schw-kern-toepl}

\subsection*{The algebra $\Fa_0$} \label{sec:algebra-fa_0} 
We briefly recall the definition and properties of the algebra $\Fa_0$. The reader is referred to  \cite{oim_symp_quant} for a more detailed exposition. $\Fa_0$ depends on the data $M$, $L$, $A$, $j$. By definition, $\Fa_0$ consists of families $( P_k : \Ci (M , L^k \otimes A ) \rightarrow \Ci ( M , L^k \otimes A),\; k \in \N^*)$ of operators whose Schwartz kernels are smooth and of the following form: for any $N$, we have uniformly on $M^2$
\begin{gather} \label{eq:devas}
 P_k ( x, y)  = \Bigl( \frac{k}{2 \pi }\Bigr) ^n E^k(x,y)   \sum _{\ell \in \Z \cap [ -N, N/2 ]} k^{-\ell} f_\ell (x,y)   +  \bigo ( k^{n- (N+1)/2} ) . 
\end{gather} 
where $E$ and the $f_{\ell}$'s are sections of $E \boxtimes \con{E}$ and $A \boxtimes \con{A}$ respectively which satisfy the following conditions. For any $x, y \in M$,  $E(x,x) = 1$, $|E(x,y)| <1 $ if $x \neq y$ and $\con{E}( x,y) = E(y,x)$. For any $ Z \in \Ci  ( M , T^{1,0}M)$,  $\nabla_{(\overline{Z}, 0)} E $ vanishes to second order along the diagonal of $M^2$. For any negative $\ell$, $f_{\ell}$ vanishes to order $-3 \ell$ along the diagonal. It is also required that the successive derivatives of $P_k (x,y)$ are uniformly slowly increasing as $k$ tends to $\infty$, cf. Section 2.2 of \cite{oim_symp_quant} for a more precise formulation. 

For any $m \in \N$, define $\Fa_m$ as the subspace of $ \Fa_0$ consisting of families with a Schwartz kernel in $\bigo ( k^{n-m/2})$. By Theorem 3.3 of \cite{oim_symp_quant}, $\Fa_0$ is an algebra and $\Fa_m . \Fa_p \subset \Fa_{m+p}$ for any $m$ and $p$. Furthermore, we can describe the quotients $\Fa_m / \Fa_{m+1}$ by a convenient symbol and compute the corresponding products as follows. First $(P_k) \in \Fa_m$ if and only if in the asymptotic expansion (\ref{eq:devas}), for any $\ell$ such that $-m \leqslant \ell \leqslant m/2$, the coefficient $f_{\ell}$ vanishes to order $m - 2 \ell$ along the diagonal. If it is the case, the symbol of $(P_k)$ is defined by 
$$ \si_m ( P_k) = \sum_{\ell \in \Z \cap [-m, m/2]}  \hb^\ell [f_{\ell}] $$ 
where $[f_\ell] \in \Ci ( M , S^{m-2\ell} ( T^* M ) \otimes \op{End} A) $ is the linearization of $f_{\ell}$ along the diagonal at order $m- 2 \ell$. More explicitly, if $\partial_1, \ldots, \partial_n$ is a local frame of $T^{1,0} M$ and $(z_i)$ is the dual frame of $(T^{1,0} M)^*$, we set
\begin{gather} \label{eq:symb}
 [f_\ell]  ( z, \con{z}) = \sum_{|\al| + |\be| = m - 2 \ell } \frac{1}{\al ! \be !} \bigl( ( \con{\partial}^\be \boxtimes \partial^\al)
 f_\ell \bigr) |_{\op{diag} M} z^ \al \con{z}^\be .
\end{gather}
Then clearly $\si_m ( P_k) = 0$ if and only if $(P_k) \in \Fa_{m+1}$. Furthermore if $(P_k) \in \Fa_p$ and $(Q_k) \in \Fa_q$ then the symbol of $(P_kQ_k) \in \Fa_{p+q}$ is equal to $\si_p (P) \star \si_q (Q)$ where $\star$ is given by
\begin{gather} \label{eq:prod_formal}
 (e \star g) (\hb, z, \con{z} ) =  \Bigr[  \exp ( \hb \De ) \bigl( e ( \hb, -u, \con{z} + \con{u} ) g ( \hb , z + u , - \con {u} ) \bigr) \Bigl]_{u = \con{u} = 0 }
\end{gather}
In this formula,  $\Delta = \sum \partial_i \con{\partial}_i $ acts on the variables $u, \bar{u}$.

\subsection*{Toeplitz operators} 

Denote by $\Pi_k$ the orthogonal projector of $\Ci ( M , L^k \otimes A)$ onto $\mathcal{H}_k$. Recall that the family $(\mathcal{H}_k)$ is chosen in such a way that $( \Pi_k)$ belongs to $\Fa_0$ and has symbol $\sip ( \Pi) = 1_A$. We have the following characterization of Toeplitz operators: 
$$T \in \To \quad \Leftrightarrow \quad T \in \Fa_0 \text{ and } \Pi T \Pi = T.$$ 
This in particular gives a description of the Schwartz kernel of a Toeplitz operator. 
Furthermore, for any  $T \in \To$ and $m \in \N$, $\si_{\op{cont}} (T) = \bigo( \hb ^m) $ if and only if $T \in \Fa_{2m}$. If it is the case, we have $ \si_{2m} (T) = \hb^m f$ and  $\si_{\op{cont}} ( T) = \hb ^m f + \bigo ( \hb^{m+1})$ for the same section $f \in \Ci ( M ,\op{End } A)$.  Another useful property is that for any odd $p$, $\Fa_p \cap \To = \Fa_{p+1} \cap \To$.

Introduce a vector field $X$ of $M$ and a derivative $D_X : \Ci ( M , A) \rightarrow \Ci ( M, A)$ in the direction of $X$  preserving the metric of $A$. Denote by $P_{X,k}$ the derivative
$$ P_{X,k} = ( \nabla^{L^k}_X \otimes \op{id} + \op{id} \otimes D_X) : \Ci ( M , L^k \otimes A) \rightarrow \Ci ( M , L^k \otimes A ) . $$
\begin{lemme} \label{lem:1}
The operator family $( \frac{i}{k} P_{X,k} \Pi_k)$ belongs to $\Fa_1 $. Its symbol is the section $\tau_Y 1_A$ where $\tau_Y \in \Ci ( M, T^*M)$ is given by $ \tau_Y = \om ( \cdot , Y^{1,0})$.
\end{lemme} 

\begin{proof} 
The Schwartz kernel of $ \frac{i}{k} P_{X,k} \Pi_k $ being $( \frac{i}{k} P_{X,k} \boxtimes \op{id} )\Pi_k$, the result is a particular case of Lemma 2.19 in \cite{oim_symp_quant}. 
\end{proof}

\begin{lemme}  \label{lem:2}
Assume that $X$ is the Hamiltonian vector field of $f \in \Ci (M)$. 
Then the family $( \bigl[  f + \frac{i}{k} P_{X,k} , \Pi_k  \bigr] )$ belongs to $\Fa_2 $. Using the same notations as in Equation (\ref{eq:symb}), its symbol is
$$ \Bigl( \sum_{i,j = 1} ^n \om ( \con{\partial}_i, [ \con{\partial}_j,X] ) \con{z}_i \con{z}_j - \om( \partial_i, [\partial_j, X]) z_i z_j \Bigr) 1_A  .$$
\end{lemme} 

This is a particular case of Lemma 3.5  in \cite{oim_symp_quant}. 
Let us deduce a first consequence of Lemma \ref{lem:1}. 
Consider a second vector field $Y$ of $M$ and a derivation $D_Y$ of $\Ci( M, A)$. Let $P_{Y, k}$ be the corresponding operator.

\begin{lemme} \label{lem:3}
The operator family $( \frac{1}{k^2} \Pi_k P_{X,k} P_{Y,k} \Pi_k)$ belongs to $\Fa_2 $. Its symbol is $ -i \hb \om ( X^{0,1}, Y^{1,0})$. 
\end{lemme} 

\begin{proof} 
By Lemma \ref{lem:1},  $( \frac{i}{k} P_{X,k} \Pi_k)$ belong to $\Fa_1 $ and its symbol is $\tau_X 1_A$. Taking adjoint and using that $(\tfrac{i}{k} \Pi_k \op{div} X ) \in \Fa_2 $, we get that $ (\Pi_k \frac{i}{k}  P_{X,k}  )$ belongs  to $\Fa_1 $. Furthermore its symbol is $ \con{\tau}_X 1_A$. Consequently, $( - \frac{1}{k^2} \Pi_k P_{X,k} P_{Y,k} \Pi_k)$  belongs  to $\Fa_2 $. By Equation (\ref{eq:prod_formal}), its symbol is 
$$ \con{\tau}_X 1_A \star \tau_Y 1_A = \hb \sum \om ( \partial_i, X^{0,1} ) \om ( \con{\partial}_i , Y^{1,0}) $$
where $(\partial_i)$ is an orthonormal frame of $T^{1,0}M$. Since $i X^{0,1}= \sum \om ( \partial_i, X^{0,1} ) \con{\partial}_i$, we obtain that $ \con{\tau}_X 1_A \star \tau_Y 1_A = i \hb \om ( X^{0,1}, Y^{1,0})$. 
\end{proof} 



\section{Proof of Theorem \ref{theo:quant}}  \label{sec:proof}

Consider $ D \in \mathcal{D} $ as in section \ref{sec:metapl-corr}. 
Recall that for any function $f \in \Ci ( M)$, $ Q_k^D ( f) =$ $ \Pi_k \bigl( f + \tfrac{i}{k } P_{f,k}  \bigr) \Pi_k$  where $P_{f,k} =  \nabla^{L^k}_{X_f} \otimes \op{id} + \op{id} \otimes D_{X_f}$.

The fact that $Q_k^D(f)$ is a Toeplitz operator with principal symbol $f$ was already observed in \cite{oim_symp_quant}. Let us recall the proof since it is an easy consequence of Lemma \ref{lem:1}. Since $( \frac{i}{k} P_{X,k} \Pi_k)$ belongs to $\Fa_1 $, the same holds for $(\Pi_k \frac{i}{k} P_{X,k} \Pi_k)$. So $(\Pi_k \frac{i}{k} P_{X,k} \Pi_k)$ is a Toeplitz operator in $\To \cap \Fa_1 = \To \cap \Fa_2$. Consequently $$Q^D_k(f) = \Pi_k f \Pi_k \mod \To \cap \Fa_2,$$ so that $Q^D(f)$ is a Toeplitz operator with principal symbol $f$. 
 This shows the first assertion of Theorem \ref{theo:quant}. The third assertion of Theorem \ref{theo:quant} has already been proved in Theorem 5.8 of \cite{oim_symp_quant}. It remains to prove the second assertion.

\begin{lemme} \label{lem:4} 
For any $f, g \in \Ci ( M )$, we have
$$\Pi_k \bigl(  f + \tfrac{i}{k} P_{f,k} \bigr) \bigl( g + \tfrac{i}{k} P_{g,k} \bigr) \Pi_k = \Pi_k \bigl( fg + \tfrac{i}{k} P_{fg, k } +  \tfrac{i}{2k}  X. g  \bigr)\Pi_k  \mod \Fa_4 \cap \To.$$
with $X$ the Hamiltonian vector field of $f$. 
\end{lemme} 
\begin{proof} Let $X$ and $Y$ be the Hamiltonian vector fields of $f$ and $g$ respectively. Then the Hamiltonian vector field of $fg$ is $f Y + g X$. Using Condition (\ref{eq:der2}), we obtain that
\begin{xalignat*}{2} P_{fg, k } = & f P_{g,k} + g P_{f, k } + \tfrac{1}{2} \bigl( df ( Y^{1,0}) + dg ( X^{1,0})  \bigr)  \\
 = &  f P_{g,k} + g P_{f, k } + \tfrac{1}{2} \bigl( df ( Y^{1,0}) + X.g - dg ( X^{0,1}) \bigr) \\
= &  f P_{g,k} + g P_{f, k } +  df ( Y^{1,0})  + \tfrac{1}{2} X.g 
\end{xalignat*}
where we have used that $$ - dg ( X^{0,1}) = \om ( X^{0,1}, Y) = \om  ( X^{0,1}, Y^{1,0}) =  \om  ( X,  Y ^{1,0} ) = df(Y ^{1,0})$$ because $T^{1,0}M$ and $T^{0,1} M$ are Lagrangian. 
So we have on the one hand that 
\begin{gather} \label{eq:onehand}
 fg + \tfrac{i}{k} P_{fg, k } +  \tfrac{i}{2k}  X.g  = fg + \tfrac{i}{k} \bigl( f P_{g,k} + g P_{f,k} + X.g \bigr) + \tfrac{i}{k} df ( Y^{1,0}) .
\end{gather}
On the other hand, we have
\begin{gather} \label{eq:otherhand}
\bigl( f + \tfrac{i}{k} P_{f,k} \bigr) \bigl( g + \tfrac{i}{k} P_{g,k} \bigr) =  fg  + \tfrac{i}{k} \bigr( g P_{f,k} + f P_{g,k} +  X. g \bigr) - \tfrac{1}{k^2} P_{f,k} P_{g, k}
\end{gather}
Clearly $ \Pi  \tfrac{i}{k} df ( Y^{1,0})\Pi$ belongs to $\Fa_2 $. Its symbol is $i \hb  df ( Y^{1,0})$. By Lemma \ref{lem:3}, $ (-\Pi_k \tfrac{1}{k^2} P_{f,k} P_{g, k}  \Pi_k ) \in \Fa_2 $ and has the same symbol. 
So
$$ \Pi_k  \tfrac{i}{k} df ( Y^{1,0})\Pi_k = -\Pi_k \tfrac{1}{k^2} P_{f,k} P_{g, k} \Pi_k  \mod \Fa_3.$$
Recall that $\Fa_3 \cap \To = \Fa_4 \cap \To$. So Equations (\ref{eq:onehand}) and (\ref{eq:otherhand}) imply the result. 
\end{proof}

\begin{theo} 
For any $f, g \in \Ci ( M )$, we have
$$ Q_k^D ( f) Q_k^D ( g) \equiv Q_k^D ( fg )  + \tfrac{1}{2ki} Q_k^D ( \{ f, g \} ) \mod \Fa_4 \cap \To$$
where $\{ f, g \}$ denotes the Poisson bracket of $f$ and $g$. 
\end{theo}

\begin{proof} Introduce the operators
$$S(f,g,k) =  \Pi_k [ f + \tfrac{i}{k} P_{f,k} , \Pi_k ] [ g + \tfrac{i}{k} P_{g,k} , \Pi_k ]\Pi_k  $$
By Lemma \ref{lem:2}, the family $( S(f,g,k))$ belongs to $ \Fa_4 \cap \To$.
A straightforward computations shows that
\begin{gather} \label{eq:sfgk}
 S(f,g,k) = Q_k^D ( f) Q_k^D ( g) - \Pi_k \bigl(  f + \tfrac{i}{k} P_{f,k} \bigr) \bigl( g + \tfrac{i}{k} P_{g,k} \bigr) \Pi_k
\end{gather}
So we have 
\begin{gather} 
   \Pi_k \bigl(  f + \tfrac{i}{k} P_{f,k} \bigr) \bigl( g + \tfrac{i}{k} P_{g,k} \bigr) \Pi_k =  Q_k^D ( f) Q_k^D ( g)  \mod \Fa_4 \cap \To .
\end{gather} 
We conclude with Lemma \ref{lem:3}. 
\end{proof}

\begin{lemme} \label{lem:5}
For any $f, g \in \Ci (M)$, we have
\begin{xalignat*}{2}   
[ f + \tfrac{i}{k} P_{f,k} , g + \tfrac{i}{k} P_{g,k} ] = & \tfrac{1}{ik} \bigl( \{ f, g \} + \tfrac{i}{k} P_{\{ f, g \} , k} \bigr) \\ &  - \tfrac{1}{k^2}  \op{id} \otimes \bigl( i R  (X,Y) + \tfrac{1}{2} B_j (X,Y) \bigr)  
\end{xalignat*} 
where $R$ is the 2-form defined in Proposition \ref{prop:met_com}, $X$ and $Y$ are the Hamiltonian vector fields of $f$ and $g$,  and $B_j (X,Y)$ is the function defined in Lemma \ref{lem:com}.
\end{lemme}

\begin{proof} 
A famous computation shows that 
$$ [ f + \tfrac{i}{k} \nabla^{L^k}_X , g + \tfrac{i}{k} \nabla^{L^k}_Y ] =  \tfrac{i}{k} \bigl( - \{ f, g \} +  \tfrac{i}{k} \nabla^{L^k}_{[X,Y]} \bigr) $$
where $X$ and $Y$ the Hamiltonian vector fields of $f$ and $g$. The result is now a consequence of Proposition \ref{prop:met_com}. 
\end{proof}

\begin{lemme} \label{lem:6}
For any $f, g \in \Ci (M)$, $ S(f,g,k) - S(g,f,k) $ belongs to $\Fa_4$ and its symbol is $\frac{\hb^2}{2}  B_j ( X,Y)$.   
\end{lemme}

\begin{proof}
By Lemma \ref{lem:2}, the symbol $\si_X$ of $([ f + \tfrac{i}{k} P_{f,k}, \Pi_k])$ has the form $ f_X(z) - \con{f_X(z)}$ where $f_X$ is quadratic. Using this fact, a direct computation shows that
$$ 1_A \star \si_X \star \si_Y \star 1_A = \tfrac{\hb^2}{2} \sum_{i,j} (\partial_i \partial_j \si_X) ( \con{\partial}_i \con{\partial}_j \si_Y) .$$
where we denote by $\star$ the product given in Equation (\ref{eq:prod_formal}). 
By  Lemma \ref{lem:2}, 
$$ (\partial_i \partial_j \si_X) = - \om ( \partial_i , [\partial_j , X]), \qquad  (\con{\partial}_i \con{\partial}_j \si_Y) =  \om ( \con{\partial}_i , [\con{\partial}_j , Y]) .$$
Since $(\partial_i)$ is an orthonormal frame, the dual frame is given by $\te_i = \tfrac{1}{i} \om ( \cdot, \con{\partial}_i)$. Since $Y$ preserves $\om$, we get $ \mathcal{L}_Y \te_i = \tfrac{1}{i} \om ( \cdot, [ Y, \con{\partial}_i] ) $. Replacing $Y$ by $X$ and taking the conjugate, $ \mathcal{L}_X \con{\te}_i = -\tfrac{1}{i} \om ( \cdot, [ X, \partial_i] ) $. Consequently
$$  1_A \star \si_X \star \si_Y \star 1_A = -  \tfrac{\hb^2}{2} \sum_{i,j} (\mathcal{L}_X \con{\te}_i)( \partial_j) (\mathcal{L}_Y \te_i) ( \con{\partial}_j)
$$ 
Antisymmetrising, we get the final result. 
\end{proof} 

\begin{theo} 
For $f,g \in \Ci (M)$, we have that 
$$ [Q_k^D (f), Q_k^D ( g) ] = \tfrac{1}{ik} \bigl( Q_k^D ( \{ f, g \} ) + \tfrac{1}{k} \Pi_k R (X,Y) \Pi_k \bigr)  \mod \Fa_6 \cap \To$$
where $X$ and $Y$ are the Hamiltonian vector fields of $f$ and $g$. 
\end{theo}

\begin{proof}
By  Equation (\ref{eq:sfgk}), 
$$   S(f,g,k) - S(g,f,k) = [Q_k^D (f), Q_k^D ( g) ] - \Pi_k [ f + \tfrac{i}{k} P_{f,k} , g + \tfrac{i}{k} P_{g,k} ]\Pi_k $$ 
By Lemma \ref{lem:6}, the left hand side is equal to $\frac{1}{k^2} \Pi_k \frac{1}{2} B_j (X,Y) \Pi_k$ modulo $\Fa_6 \cap \To$. By  Lemma \ref{lem:5}, the right hand side is equal to 
$$  [Q_k^D (f), Q_k^D ( g) ] - \tfrac{1}{ik}  Q_k^D ( \{ f, g \}  + \tfrac{i}{k^2} \Pi_k R (X,Y) \Pi_k  + \tfrac{1}{k^2}  \Pi_k \tfrac{1}{2} B_j (X,Y) \Pi_k  $$
modulo $\Fa_6 \cap \To$. The result follows. 
\end{proof}

\bibliographystyle{alpha}
\bibliography{bibsub}

\end{document}